\documentclass[12pt]{amsart}
\usepackage{amssymb,color}
\usepackage{tikz}
\theoremstyle{plain}
\newtheorem{theorem}{Theorem}

\newtheorem{lem}{Lemma}
\newtheorem{prop}{Proposition}

\theoremstyle{definition}
\newtheorem{exmp}{Example}

\theoremstyle{remark}
\newtheorem{rem}{Remark}

\allowdisplaybreaks

\long\def\symbolfootnote[#1]#2{\begingroup
	\def\thefootnote{\fnsymbol{footnote}}\footnote[#1]{#2}\endgroup}

\begin{document}
	\title[Nonlinear multivalued Duffing systems]
	{Nonlinear multivalued Duffing systems}
	\author{Nikolaos S. Papageorgiou, Calogero Vetro, Francesca Vetro}
	\address[N.S. Papageorgiou]{Department of Mathematics, National Technical University, Zografou campus, 15780, Athens, Greece}
	\email{npapg@math.ntua.gr}
	\address[C. Vetro]{Department of Mathematics and Computer Science, University of Palermo, Via Archirafi 34, 90123, Palermo, Italy}
	\email{calogero.vetro@unipa.it}
	\address[F. Vetro]{$^{(1)}$Nonlinear Analysis Research Group, Ton Duc Thang University, Ho Chi Minh City, Vietnam \\ $^{(2)}$ Faculty of Mathematics and Statistics, Ton Duc Thang University, Ho Chi Minh City, Vietnam}
	\email{francescavetro@tdt.edu.vn}


	
	\maketitle

\begin{abstract}
We consider a multivalued nonlinear Duffing system driven by a nonlinear nonhomogeneous  differential operator. We prove existence theorems for both the convex and nonconvex problems (according to whether the multivalued perturbation is convex valued or not). Also, we show that the solutions of the nonconvex problem are dense in those of the convex (relaxation theorem). Our work extends the recent one by Kalita-Kowalski (JMAA,  https://doi.org/10.1016/j.jmaa. 2018.01.067).
\end{abstract}

\section{Introduction}
In this paper we study the following nonlinear multivalued Duffing system
\begin{equation}
\label{eq1} \begin{cases} - a(u^\prime(t))^\prime -r(t)|u^\prime(t)|^{p-2}u^\prime(t)\in F(t,u(t)) & \mbox{for a.a. } t \in T=[0,b],\\ u(0)=u(b)=0, \quad 1<p<+\infty. &\end{cases}
\end{equation}

In this system $a: \mathbb{R}^N \to \mathbb{R}^N$ is a suitable monotone homeomorphism which incorporates as special cases many differential operators of interest such as the vector $p$-Laplacian which corresponds to the map $a(y)=|y|^{p-2}y$ for all $y \in \mathbb{R}^N$ ($1<p<+\infty$). The term $F(t,x)$ is a multivalued perturbation. We prove existence theorems for both the ``convex problem'' (that is, $F$ is convex valued) and the ``nonconvex problem'' (that is, $F$ has nonconvex values). Finally we show that under more restrictive conditions on the data of the problem, the solutions of the nonconvex problem are dense in those of the convex problem (relaxation theorem).

The starting point of our work here is the recent paper of Kalita-Kowalski \cite{Ref7}, where $N=1$ (scalar problem), $a(y)=y$ for all $y \in \mathbb{R}^N$ (semilinear equation), the growth condition on $F(t,\cdot)$ is more restrictive and the authors treat only the convex problem.

The Duffing equation originates as a model of certain damped and driven oscillators. The equation is well-known for its chaotic behavior, well documented in the works of Holmes \cite{Ref4} and Moon-Holmes \cite{Ref12}. Additional recent results on the scalar, semilinear and single-valued version of the equation, can be found in Galewski \cite{Ref2}, Kowalski \cite{Ref8}, Tomiczek \cite{Ref14}.

\section{Mathematical Background}

The presence of the term $r(t)|u^\prime|^{p-2}u^\prime$ makes problem \eqref{eq1} nonvariational and so our approach is topological based on the fixed point theory. Our tools come from multivalued analysis and nonlinear functional analysis.

Let $X$ be a Banach space. We introduce the following hyperspaces:

\begin{align*}
P_{f(c)}(X)&=\{A \subseteq X :  \mbox{nonempty, closed (and convex)}\},\\
P_{(w)k(c)}(X)&=\{A \subseteq X :  \mbox{nonempty, (weakly-) compact (and convex)}\}.
\end{align*}
A multifunction (set-valued function) $G: X \to 2^X \setminus \{\emptyset\}$ is said to be ``upper semicontinuous (usc)'' (resp. ``lower semicontinuous (lsc)"), if for all $C \subseteq X$ closed, the set $G^-(C)=\{x \in X : G(x) \cap C \neq \emptyset \}$ (resp. $G^+(C)=\{x \in X : G(x)  \subseteq C\}$) is closed.

Given $A,B \in P_f(X)$, we set 
$$h(A,B)= \max\{\sup_{a \in A} d(a,B), \sup_{b \in B} d(b,A)\}. $$

Then $h(\cdot,\cdot)$ is an extended metric on $P_f(X)$ known as the ``Hausdorff metric''.

Now let $(\Omega,\Sigma)$ be a measurable space and $X$ a separable Banach space. A multifunction $F: \Omega \to 2^X \setminus \{\emptyset\}$ is said to be ``graph measurable'' if $ {\rm Gr }F=\{(\omega,x) \in \Omega \times X: x \in F(\omega)\} \in \Sigma \otimes B(X)$ with $B(X)$ being the Borel $\sigma$-field of $X$. Suppose that $\mu(\cdot)$ is a  finite measure on $\Sigma$. By $S^1_F$ we denote the set of $L^1(\Omega,X)$-selectors of $F(\cdot)$, that is, $S_F^1=\{f \in L^1(\Omega,X) : f(\omega) \in F(\omega) \, \, \mu\mbox{-a.e. on }\Omega\}$. Note that as a consequence of the Yankov-von Neumann-Aumann selection theorem (see Hu-Papageorgiou \cite{Ref5}, Theorem 2.14, p. 158), we have that for a graph measurable multifunction $F(\cdot)$, the set $S^1_F$ is nonempty if and only if the function $\omega \to \inf [\|x\| : x \in F(\omega)]$ belongs in $L^1(\Omega)$. The set $S^1_F$ is ``decomposable'', that is
$$\mbox{``if $(A,f_1,f_2) \in \Sigma \times S_F^1 \times S_F^1$, then
$\chi_A f_1 + \chi_{\Omega \setminus A} f_2 \in  S_F^1.$''}$$

Here for any $C \subseteq \Omega$, $\chi_C$ denotes the characteristic function of $C$, that is, $$\chi_C(\omega)=\begin{cases}1 & \mbox{if } \omega \in C\\0 & \mbox{if } \omega \not \in  C\end{cases}.$$

Let $Y,V$ be Banach spaces and $K:Y \to V$. We say that $K(\cdot)$ is ``completely continuous'', if $y_n \xrightarrow{w}y$ in $Y$, then $K(y_n)\to K(y)$ in $V$. A multivalued map $G:Y \to 2^V \setminus \{\emptyset\}$ is said to be ``compact'', if it is usc and maps bounded sets in $Y$ to relatively compact sets in $V$.

We will need the following multivalued generalization of the Leray-Schauder alternative theorem, due to Bader \cite{Ref1}. So, $Y$, $V$ are Banach spaces, $N:Y \to P_{wkc}(V)$ is usc from $Y$ into $V_w$ (= the Banach space $V$ endowed with the weak topology) and $K: V \to Y$ is completely continuous. We set $G=K \circ N$.

\begin{prop}\label{prop1}
If $Y,V,G$ are as above and $G(\cdot)$ is a compact multifunction, then one of the following statements holds: 
\begin{itemize}
\item[$(a)$] the set $S=\{y \in Y: y \in \lambda G(y), \, 0<\lambda<1\}$ is unbounded;
\item[$(b)$] $G(\cdot)$ admits a fixed point (that is, there exists $y \in Y$ such that $y \in G(y)$).
\end{itemize}
\end{prop}

Consider the following nonlinear vector eigenvalue problem 
\begin{equation}
\label{eq2} \begin{cases}  -(|u^\prime(t)|^{p-2}u^\prime(t))^\prime =\widehat{\lambda}|u(t)|^{p-2}u(t) & \mbox{for a.a. } t \in T,\\ u(0)=u(b)=0, \quad 1<p<+\infty. &\end{cases}
\end{equation}

We say that $\widehat{\lambda} \in \mathbb{R}$ is an eigenvalue, if problem \eqref{eq2} admits a nontrivial solution $\widehat{u} \in W_0^{1,p}((0,b),\mathbb{R}^N)$, known as an eigenfunction corresponding to $\widehat{\lambda}$. We know that \eqref{eq2} has a sequence of eigenvalues
$$\widehat{\lambda}_n=\left(\frac{n}{b}\right)^p(p-1)\left[2\int_0^1\frac{dt}{(1-t^p)^{1/p}}\right]^p \quad \mbox{for all } n \in \mathbb{N}$$
and the corresponding eigenfunctions are
$$\widehat{u}_n(t)=a u_n(t)\quad \mbox{for all } n \in \mathbb{N}$$
with $a \in \mathbb{R}^N$ and $u_n$ being the corresponding scalar Dirichlet eigenfunctions. Therefore $|\widehat{u}_n(t)| \neq 0$ for a.a. $t \in T$. Also, we have
\begin{equation}
\label{eq3}\widehat{\lambda}_1 = \inf \left[\frac{\|u^\prime\|_p^p}{\|u\|_p^p}: u \in W_0^{1,p}((0,b),\mathbb{R}^N), \, u \neq 0\right]
\end{equation}
(see Gasi\'nski-Papageorgiou \cite{Ref3}, p. 768).

If by $\|\cdot\|$ we denote the norm of the Sobolev space $W_0^{1,p}((0,b),\mathbb{R}^N)$, then from the Poincar\'e inequality, we have
$$\|u\|=\|u^\prime\|_p \quad \mbox{for all } u \in W_0^{1,p}((0,b),\mathbb{R}^N).$$

For notational economy in the sequel we will write
\begin{align*}
C_0 &=C_0(T, \mathbb{R}^N)=\{u \in C(T, \mathbb{R}^N): u(0)=u(b)=0\},\\
C^1_0 &=C^1_0(T, \mathbb{R}^N)= C^1(T, \mathbb{R}^N)\cap C_0(T, \mathbb{R}^N),\\
 W_0^{1,p}&=  W_0^{1,p}((0,b),\mathbb{R}^N),
 \\
 W^{1,p}&=  W^{1,p}((0,b),\mathbb{R}^N),\\
 L_N^r&=L^r(T, \mathbb{R}^N) \quad \mbox{for all }1 \leq r \leq +\infty.
\end{align*}

\section{Convex Problem}\label{S3}
In this section we deal with the ``convex problem'' (that is, $F$ is convex valued). The precise hypotheses on the data of the convex problem are the following:

\begin{itemize}
\item[$H(a)$:] $a: \mathbb{R}^N \to \mathbb{R}^N$ is a continuous and strictly monotone map such that 
$$c_0|y|^p \leq (a(y),y)_{\mathbb{R}^N} \quad \mbox{for all $y \in \mathbb{R}^N$, some  $c_0>0$}.$$
\end{itemize}

\begin{rem} Note that $a(\cdot)$ is maximal monotone, surjective (see Gasi\'nski-Papageorgiou \cite{Ref3} (pp. 309, 319)). Moreover $|a^{-1}(y)| \to +\infty$ as $|y| \to +\infty$. We stress that no growth restriction is imposed. Such very general operators were first used by Man\'asevich-Mawhin \cite{Ref10,Ref11}. Later Kyritsi-Matzakos-Papageorgiou \cite{Ref9} used them in the context of multivalued systems with unilateral constraints.
\end{rem}

\begin{exmp} The following maps $a: \mathbb{R}^N \to \mathbb{R}^N$ satisfy hypotheses $H(a)$:
\begin{align*}
a(y)&=|y|^{p-2}y \quad \mbox{for all $y \in \mathbb{R}^N$, with  $1 < p < +\infty$}\\
&\hskip .5cm \mbox{(this map corresponds to the vector $p$-Laplacian)},\\
a(y)&=|y|^{p-2}y +|y|^{q-2}y\quad \mbox{for all $y \in \mathbb{R}^N$, with $1 < q<p < +\infty$}\\
&\hskip .5cm \mbox{(this map corresponds to the vector $(p,q)$-Laplacian)},\\
a(y)&=(1+|y|^{2})^\frac{p-2}{2}y \quad \mbox{for all $y \in \mathbb{R}^N$, with $1 < p < +\infty$},\\
a(y)&=|y|^{p-2}y (ce^{|y|^p}-1)\quad \mbox{for all $y \in \mathbb{R}^N$, with $1 < p < +\infty$, $c>1$}.
\end{align*}
\end{exmp}

\begin{itemize}
	\item[$H(r)$:]$r \in L^\infty(T)$.
\end{itemize}

\begin{itemize}
\item[$H(F)_1$:] $F:T  \times \mathbb{R}^N \to P_{kc}(\mathbb{R}^N)$ is a multifunction such that
\begin{enumerate}
\item[(i)] for all $x \in \mathbb{R}^N$, the multifunction $t \to F(t,x)$ admits a measurable selection;
\item[(ii)] for a.a. $t \in T$,  $ {\rm Gr }F(t,\cdot) \subseteq \mathbb{R}^N \times \mathbb{R}^N$ is closed;
\item[(iii)] if $\xi =c_0 -\dfrac{\|r\|_\infty}{\widehat{\lambda}_1^{1/p}}>0$, then there exists a function $\theta \in L^\infty(T)_+$ such that $\theta(t) \leq \widehat{\lambda}_1 \xi$ for a.a. $t \in T$, the inequality is strict on a set of positive measure, 
$$\limsup_{|x|\to+\infty}\frac{\sup[(h,x)_{\mathbb{R}^N}: h \in F(t,x)]}{|x|^p}\leq \theta(t) \quad \mbox{uniformly for a.a. $t \in T$}$$
and for every $\eta>0$, there exists $a_\eta \in L^1(T)_+$ such that
$$|F(t,x)| =\sup [|h|:h\in F(t,x)] \leq a_\eta (t) \quad \mbox{for a.a. $t \in T$, all $|x| \leq \eta$}.$$
\end{enumerate}\end{itemize}

\begin{rem}
Hypothesis $H(F)_1$(i) is satisfied if, for example, for all $x \in \mathbb{R}^N$ the multifunction $t \to F(t,x)$ is graph measurable (that is, ${\rm Gr}F(\cdot,x) \in \mathcal{L}_T \otimes B(\mathbb{R}^N)$ for all $x \in \mathbb{R}^N$ with $\mathcal{L}_T$ being the lebesgue $\sigma$-field of $T$ and $ B(\mathbb{R}^N)$ is the Borel $\sigma$-field of  $\mathbb{R}^N$).
\end{rem}

Let $\psi:W_0^{1,p} \to \mathbb{R}$ be the $C^1$-functional defined by
$$\psi(u)=\xi\|u^\prime\|_p^p -\int_0^b \theta(t)|u|^p dt \quad \mbox{for all }u \in W_0^{1,p} .$$

We have that $\psi(u) \geq0$ for all $u \in W_0^{1,p}$. To see this, let $u \in W_0^{1,p}$. Then
\begin{align*}
\psi(u)&= \xi \|u^\prime\|_p^p -\int_0^b \theta(t)|u|^pdt \\ & \geq \xi \widehat{\lambda}_1\|u\|_p^p-\int_0^b\theta(t)|u|^pdt \quad \mbox{(see \eqref{eq3})}\\ & = \int_0^b[\xi \widehat{\lambda}_1-\theta(t)]|u|^p dt \geq 0\quad \mbox{(see hypothesis $H(F)_1$(iii))}.
\end{align*}

\begin{lem}\label{lem2}
	There exists $c_1>0$ such that $c_1 \|u\|^p \leq \psi(u)$ for all $u \in W_0^{1,p}$.
\end{lem}
\begin{proof}
	We argue indirectly. So, suppose that the lemma is not true. Exploiting the $p$-homogeneity of $\psi(\cdot)$, we can find $\{u_n\}_{n \geq 1} \subseteq W^{1,p}_0$ such that
\begin{equation}
\label{eq4} \|u_n\|=1 \mbox{ and } 0 \leq \psi(u_n)<\frac{1}{n} \quad \mbox{for all $n \in \mathbb{N}$}.
\end{equation}	

We may assume that
$$u_n \xrightarrow{w} u \mbox{ in } W_0^{1,p}.$$

Also note that $\psi(\cdot)$ is sequentially weakly lower semicontinuous (recall that $
W_0^{1,p} \hookrightarrow C_0$ compactly). So, in the limit as $n \to +\infty$, we have 
\begin{align}
\nonumber \psi(u)&=0,\\ \label{eq5}\Rightarrow \quad & \xi \|u^\prime\|_p^p=\int_0^b \theta(t)|u|^p dt \leq \widehat{\lambda}_1
\xi \|u\|_p^p,\\ \nonumber \Rightarrow \quad & \|u^\prime\|_p^p \leq \widehat{\lambda}_1
  \|u\|_p^p, \\ \nonumber \Rightarrow \quad & \|u^\prime\|_p^p = \widehat{\lambda}_1
  \|u\|_p^p \quad \mbox{(see \eqref{eq3})},\\ \nonumber \Rightarrow \quad & u=\eta \widehat{u} \quad \mbox{for some } \eta \in \mathbb{R}.\end{align} 
  
If $\eta=0$, then $u=0$ and so from \eqref{eq4} we have $\|u_n^\prime\|_p \to 0,$ which contradicts the fact that $\|u_n\|=1$ for all $n \in \mathbb{N}$. So, $\eta \neq 0$ and it follows that $|u(t)| \neq 0$ for a.a. $t \in T$. From \eqref{eq5} and the hypothesis on $\theta(\cdot)$ we have $$\|u^\prime\|_p^p < \widehat{\lambda}_1 \|u\|_p^p,$$ a contradiction to \eqref{eq3}. 
\end{proof}

Let $h \in L^1_N$ and consider the following auxiliary Dirichlet problem
\begin{equation}
\label{eq6} -a(u^\prime(t))^\prime =h(t) \quad \mbox{for a.a. $t \in T$, $u(0)=u(b)=0$.}
\end{equation}

By Lemma 4.1 of Man\'asevich-Mawhin \cite{Ref11}, we know that problem \eqref{eq6} admits a unique solution $u =K(h) \in C_0^1$. So, we can define the solution map $K: L^1_N \to C_0^1$.

\begin{prop}\label{prop3}
	The solution map $K: L^1_N \to C_0^1$ is completely continuous.
\end{prop}

\begin{proof}
	Suppose that $h_n \xrightarrow{w}h$ in $L^1_N$ and let $u_n =K(h_n)$, $n \in \mathbb{N}$. We have
	\begin{equation}
	\label{eq7} -a(u_n^\prime(t))^\prime =h_n(t) \quad \mbox{for a.a. $t \in T$, $u_n(0)=u_n(b)=0$, $n \in \mathbb{N}$.}
	\end{equation}
	
Taking inner product with $u_n(t)$, integrating over $T=[0,b]$ and performing integration by parts, we obtain 
\begin{align}
\nonumber & \int_0^b  (a(u_n^\prime),u_n^\prime)_{\mathbb{R}^N}dt= \int_0^b  (h_n,u_n)_{\mathbb{R}^N}dt\\ \nonumber \Rightarrow \quad & c_0\|u_n^\prime\|_p^p \leq \|h_n\|_1 \|u_n\|_\infty \leq c_2 \|u_n^\prime\| \quad \mbox{for some $c_2>0$, all $n \in \mathbb{N}$}\\ \nonumber & \hskip 6cm \mbox{(see hypothesis $H(a)$),}\\ \nonumber \Rightarrow \quad & \{u_n\}_{n \geq 1} \subseteq W_0^{1,p} \mbox{ is bounded,}\\\Rightarrow \quad & \{u_n\}_{n \geq 1} \subseteq C_0  \mbox{ is relatively compact} \label{eq8}\\ \nonumber & \hskip 4cm \mbox{(recall that $W_0^{1,p} \hookrightarrow C_0$ compactly).}
\end{align}	

From \eqref{eq7} we have
\begin{align}
\label{eq9} & a(u_n^\prime(t))=a(u_n^\prime(0))+\int_0^t
h_n(s)ds \quad \mbox{for all $t \in T$, all $n \in \mathbb{N}$,}\\ \nonumber
\Rightarrow \quad & u_n^\prime(t)= a^{-1}\left[ a(u_n^\prime(0))+\int_0^t
h_n(s)ds \right],\\ \nonumber
\Rightarrow \quad & 0=\int_0^b a^{-1}\left[ a(u_n^\prime(0))+\int_0^t
h_n(s)ds \right]dt\quad \mbox{for all $n \in \mathbb{N}$.}\end{align}	

Invoking Proposition 3.1 of Man\'asevich-Mawhin \cite{Ref11}, we infer that
\begin{equation}
\label{eq10} \{a(u_n^\prime(0))\}_{n\geq 1}\subseteq \mathbb{R}^N \mbox{ is bounded}.
\end{equation}

From \eqref{eq9}, \eqref{eq10} and the Arzel\`{a}-Ascoli theorem it follows that 
\begin{equation}
\label{eq11} \{a(u_n^\prime(\cdot))\}_{n\geq 1}\subseteq C(T,\mathbb{R}^N) \mbox{ is relatively compact}.
\end{equation}

Let $\widehat{a}^{-1}: C(T,\mathbb{R}^N) \to C(T,\mathbb{R}^N)$ be defined by
$$\widehat{a}^{-1}(u)(\cdot)=a^{-1}(u(\cdot)) \quad \mbox{for all } u \in C(T,\mathbb{R}^N).$$

Evidently this map is continuous and maps bounded sets to bounded sets. So, from \eqref{eq11} it follows that
\begin{equation}
\label{eq12} \{u_n^\prime\}_{n\geq 1}\subseteq C(T,\mathbb{R}^N) \mbox{ is relatively compact}.
\end{equation}
From \eqref{eq8} and \eqref{eq12} it follows that  
$$\{u_n\}_{n\geq 1}\subseteq C_0^1 \mbox{ is relatively compact}.$$

We may assume that
\begin{equation}
\label{eq13} u_n \to u \mbox{ in } C_0^1.
\end{equation}

From \eqref{eq7} we have
\begin{align*}
 & \int_0^b  (a(u_n^\prime),v^\prime)_{\mathbb{R}^N}dt= \int_0^b  (h_n,v)_{\mathbb{R}^N}dt \quad \mbox{for all $v \in W^{1,p}_0$, all $n \in \mathbb{N}$},\\ \nonumber \Rightarrow \quad &\int_0^b  (a(u^\prime),v^\prime)_{\mathbb{R}^N}dt= \int_0^b  (h,v)_{\mathbb{R}^N}dt \quad \mbox{for all $v \in W^{1,p}_0$ (see \eqref{eq13})},\\ \nonumber \Rightarrow \quad & u=K(h).
\end{align*}

Therefore $K(h_n) \to K(h)$ in $C_0^1$ and so the solution map $K(\cdot)$ is completely continuous.
\end{proof}

Let $N_1: C_0^1 \to 2^{L_N^1}$ be defined by $N_1(u)=S^1_{F(\cdot,u(\cdot))}$ for all $u \in C^1_0$.

\begin{prop}\label{prop4}
If hypotheses $H(r)$, $H(F)_1$ hold, then the multifunction $N_1(\cdot)$ is $P_{wkc}(L^1_N)$-valued and it is usc from $C_0^1$ with the norm topology into $L_N^1$ with the weak topology.
\end{prop}
\begin{proof}
	First we show that $N_1$ has values in $P_{wkc}(L^1_N)$. Clearly, hypothesis $H(F)_1$(iii) implies that for every $u \in C^1_0$, the set $N_1(u)$ is w-compact, convex. We only need to show that the set $S^1_{F(\cdot,u(\cdot))}$ is nonempty. To this end, let $\{s_n\}_{n\geq 1}$ be simple functions such that
	$$|s_n(t)| \leq |u(t)| \mbox{ and } s_n(t) \to u(t) \quad \mbox{for a.a. } t \in T.$$
	
	On account of hypothesis $H(F)_1$(i), the multifunction $t \to F(t,s_n(t))$ admits a measurable selection $f_n:T \to \mathbb{R}^N$. If $\eta=\|u\|_\infty$, then by hypothesis $H(F)_1$(iii), we have 
	$$|f_n(t)| \leq a_\eta(t) \quad \mbox{for a.a. $t \in T$, all $n \in \mathbb{N}$.}$$
	
By the Dunford-Pettis theorem, we may assume that
$$f_n \xrightarrow{w}f \mbox{ in } L^1_N.$$

Invoking Proposition 3.9, p. 694, of Hu-Papageorgiou \cite{Ref5}, we have
\begin{align*}
f(t) & \in \overline{{\rm conv}}\limsup \{f_n(t)\}_{n \geq 1}\\ & \subseteq \overline{{\rm conv}}\limsup_{n \to +\infty}F(t,s_n(t))\\ & \subseteq  F(t,u(t)) \quad \mbox{for a.a. $t \in T$ (see hypothesis $H(F)_1$(ii))},\\ \Rightarrow \quad & f \in S^1_{F(\cdot,u(\cdot))}.
\end{align*}
	
	Therefore we conclude that $$N_1(u) \in P_{wkc}(L^1_N) \quad \mbox{for all }u \in C^1_0.$$
	
To show the claimed by the proposition upper semicontinuity, according to Proposition 2.23, p. 43, of Hu-Papageorgiou \cite{Ref5}, it suffices to show that $ {\rm Gr }N_1=\{(u,f) \in C_0^1 \times L_N^1: f \in N_1(u)\}$ is sequentially closed in $C_0^1 \times (L_N^1,w)$. So, consider a sequence $\{(u_n,f_n)\}_{n\geq 1} \subseteq {\rm Gr }N_1$ and assume that
\begin{equation}
\label{eq14}u_n \to u \mbox{ in } C^1_0 \mbox{ and } f_n \xrightarrow{w}f \mbox{ in }L^1_N.
\end{equation}	

Let $\eta=\sup_{n \geq 1} \|u_n\|_\infty < +\infty$ (see \eqref{eq14}). Then by hypothesis $H(F)_1$(iii) we have
	$$|f_n(t)| \leq a_\eta(t) \quad \mbox{for a.a. $t \in T$, all $n \in \mathbb{N}$, with $a_\eta \in L^1(T)_+$.}$$
	
So, at least for a subsequence, we have 	$$ f_n \xrightarrow{w}f \mbox{ in }L^1_N.$$

As before using Proposition 3.9, p. 694, of Hu-Papageorgiou \cite{Ref5} we obtain $f \in S^1_{F(\cdot,u(\cdot))}$. Hence $(u,f) \in {\rm Gr}N_1$ and so we have the desired upper semicontinuity of $N_1(\cdot)$.
\end{proof}

Clearly the map $u \to r(\cdot)|u^\prime(\cdot)|^{p-2}u^\prime(\cdot)$ is continuous from $C_0^1$ into $L_N^1$. Therefore the multifunction $N:C_0^1 \to P_{wkc}(L^1_N)$ defined by $$N(u)=N_1(u)+r(\cdot)|u^\prime(\cdot)|^{p-2}u^\prime(\cdot) \quad \mbox{for all } u \in C_0^1,$$ is usc from $C_0^1$ into $(L_N^1,w)$.

Then Propositions \ref{prop3} and \ref{prop4} imply that $u \to K \circ N(u)$ is compact from $C_0^1$ into itself. We introduce the set
$$S=\{u \in C_0^1 : u=\lambda K N(u), \, 0< \lambda<1\}.$$

\begin{prop}
	\label{prop5} If hypotheses $H(a)$, $H(r)$, $H(F)_1$ hold, then $S \subseteq C_0^1$ is bounded.
\end{prop}

\begin{proof}
	Let $u \in S$. Then we have
\begin{align*}
& \frac{1}{\lambda} u \in K N(u),\\ \Rightarrow \quad & -a\left(\frac{1}{\lambda}u^\prime\right)^\prime -r(t)|u^\prime|^{p-2}u^\prime = f \quad \mbox{with }f \in S^1_{F(\cdot,u(\cdot))}.
\end{align*}	

We act with $u$ and perform integration by parts. We obtain
\begin{equation}
\label{eq15}\int_0^b \left(a\left(\frac{1}{\lambda}u^\prime\right),u^\prime \right)_{\mathbb{R}^N}dt -\int_0^b r(t) |u^\prime |^{p-2}(u^\prime,u)_{\mathbb{R}^N}dt= \int_0^b(f,u)_{\mathbb{R}^N}dt .
\end{equation}

We have
\begin{align}
\nonumber \left| \int_0^b r(t) |u^\prime |^{p-2}(u^\prime,u)_{\mathbb{R}^N}dt \right| & \leq \|r\|_\infty \int_0^b   |u^\prime |^{p-1}|u|dt \\ \nonumber & \leq \|r\|_\infty \|u^\prime \|_p^{p-1}\|u\|_p \quad \mbox{(by H\"{o}lder's inequality)}\\ \label{eq16} & \leq \frac{\|r\|_\infty}{\widehat{\lambda}^{1/p}} \|u^\prime \|_p^{p} \quad \mbox{(see \eqref{eq3})}.
\end{align}

Also from hypothesis $H(F)_1$(iii) we see that given $\varepsilon >0$, we can find $a_\varepsilon \in L^1(T)_+$ such that
\begin{equation}
\label{eq17} (f(t),u(t))_{\mathbb{R}^N} \leq a_\varepsilon(t)+[\theta(t)+\varepsilon]|u(t)|^p \quad \mbox{for a.a. }t \in T.
\end{equation}

We return to \eqref{eq15} and use hypothesis $H(a)$, the fact that $0<\lambda<1$ and \eqref{eq16}, \eqref{eq17}. Then
\begin{align*}
&\left[c_0 - \frac{\|r\|_\infty}{\widehat{\lambda}_1^{1/p}}\right]\|u^\prime\|_p^p \leq \|a_\varepsilon\|_1 + \int_0^b \theta(t)|u|^p dt + \frac{\varepsilon}{\widehat{\lambda}_1}\|u^\prime\|_p^p,\\ \Rightarrow \quad & \xi \|u^\prime\|_p^p - \int_0^b \theta(t)|u|^p dt-\frac{\varepsilon}{\widehat{\lambda}_1}\|u^\prime\|_p^p \leq \|a_\varepsilon\|_1,\\ \Rightarrow \quad & \left[c_1 - \frac{\varepsilon}{\widehat{\lambda}_1} \right]\|u\|^p \leq \|a_\varepsilon\|_1 \quad \mbox{(see Lemma \ref{lem2}).}
\end{align*}

Choosing $\varepsilon \in (0, \widehat{\lambda}_1c_1)$, we infer that
\begin{align}
\nonumber & S \subseteq W_0^{1,p} \mbox{ is bounded},\\ \label{eq18} \Rightarrow \quad & S \subseteq C_0 \mbox{ is compact and $\{u^\prime\}_{u \in S}\subseteq L^p_N$ is bounded.}
\end{align}

We have
\begin{equation}
\label{eq19} -a\left(\frac{1}{\lambda}u^\prime\right)^\prime =f+r|u^\prime|^{p-2}u^\prime.
\end{equation}

Let $h_u= f+r|u^\prime|^{p-2}u^\prime$, $u \in S$. From \eqref{eq18} and hypothesis $H(F)_1$(iii) it follows that $$\{h_u\}_{u \in S}\subseteq L^1_N \mbox{ is uniformly integrable.}
$$

This fact and \eqref{eq19} (recall $0<\lambda<1$), as in the proof of Proposition \ref{prop3} imply that $S \subseteq C_0^1$ is bounded (in fact relatively compact).\end{proof}

Applying Proposition \ref{prop1}, we have the following existence theorem for problem \eqref{eq1}.

\begin{theorem}
	\label{T6} If hypotheses $H(a)$, $H(r)$, $H(F)_1$ hold, then problem \eqref{eq1} admits a solution $u_0 \in C_0^1$.
\end{theorem}

\begin{rem}
	It is clear from the proof of Proposition \ref{prop5}, that under the above hypotheses the solution set of \eqref{eq1} is compact in $C_0^1$.
\end{rem}

\section{Nonconvex Problem}\label{S4}

In this section, we prove an existence theorem for the case when $F(t,x)$ has nonconvex values. Now the hypotheses on $F$ are the following:
\begin{itemize}
	\item[$H(F)_2$:] $F:T  \times \mathbb{R}^N \to P_{f}(\mathbb{R}^N)$ is a multifunction such that
	\begin{enumerate}
		\item[(i)] $(t,x) \to F(t,x)$ is graph measurable (that is ${\rm Gr}F \in \mathcal{L}_T \otimes B(\mathbb{R}^N)$);
		\item[(ii)] for a.a. $t \in T$, $x \to F(t,x)$ is lsc;
		\item[(iii)] the same as hypothesis $H(F)_1$(iii).
\end{enumerate}	\end{itemize}

\begin{theorem}
	\label{T7} If hypotheses $H(a)$, $H(r)$, $H(F)_2$ hold, then problem \eqref{eq1} admits a solution $u_0 \in C_0^1$.
\end{theorem}

\begin{proof} On account of hypothesis $H(F)_2$(i), for every $u \in C_0^1$, we have that 
\begin{align*}
& t \to F(t,u(t)) \mbox{ is graph measurable}\\ & \hskip 4cm \mbox{(see Hu-Papageorgiou \cite{Ref5}, Theorem 2.4, p. 156)}\\ \Rightarrow \quad & N_1(u)=S^1_{F(\cdot,u(\cdot))} \in P_f(L_N^1) \mbox{ for all $u \in C_0^1$} \quad \mbox{(see hypothesis $H(F)_2$(iii))}.
\end{align*}	

From Hu-Papageorgiou \cite{Ref6}, Proposition 2.7, p. 237, we have that $N_1:C_0^1 \to P_f(L^1_N)$ is lsc. Also, it has decomposable values. So, Theorem 8.7, p. 245, of Hu-Papageorgiou \cite{Ref5} implies that we can find a continuous map $\gamma : C_0^1 \to L_N^1$ such that $$\gamma(u) \in N_1(u) \quad \mbox{for all } u \in C_0^1.$$

We consider the following Dirichlet problem
$$-a(u^\prime(t))^\prime -r(t)|u^\prime(t)|^{p-2}u^\prime(t)=\gamma(u)(t) \quad \mbox{for a.a. $t \in T$, $u(0)=u(b)=0$.}$$

Then from Theorem \ref{T6} we know that this problem has a solution $u_0 \in C_0^1$. Evidently $u_0$ is also a solution of \eqref{eq1}.
	\end{proof}
	
\section{Relaxation Theorem}\label{S5}

In this section $p=2$ and we deal with the following two problems:
\begin{align}
&-a(u^\prime(t))^\prime -r(t)u(t) \in F(t,u(t)) \quad \mbox{for a.a. $t \in T$, $u(0)=u(b)=0$,}\label{eq20}\\&-a(u^\prime(t))^\prime -r(t)u(t) \in {\rm \overline{conv} } F(t,u(t)) \quad \mbox{for a.a. $t \in T$, $u(0)=u(b)=0$.}\label{eq21}
\end{align}	

By $\widehat{S} \subseteq C_0^1$ we denote the solution set of \eqref{eq20} and by $\widehat{S}_c \subseteq C_0^1$ the solution set of \eqref{eq21}. 

Under stronger conditions on the map $a(\cdot)$ and the orientor field $F$, we show that  $$\widehat{S}_c = \overline{\widehat{S}}^{C_0^1}.$$

Such a result is known as ``relaxation theorem'' and has important applications in control theory.

The new stronger conditions on the map $a(\cdot)$ are the following:

\begin{itemize}
	\item[$H(a)^\prime$:] $a: \mathbb{R}^N \to \mathbb{R}^N$ is continuous and 
	$$c_0|y|^2 \leq (a(y),y)_{\mathbb{R}^N} \quad \mbox{for all $y \in \mathbb{R}^N$, some  $c_0>0$},$$ and for every $\eta>0$, there exists $\widehat{c}_\eta>0$ such that
	$$\widehat{c}_\eta|y-v|^2 \leq (a(y)-a(v),y-v)_{\mathbb{R}^N} \quad \mbox{for all $|y|,|v|  \leq \eta$}.$$
\end{itemize}

\begin{rem} Clearly $a(\cdot)$ is strictly monotone and maximal monotone too.
\end{rem}

\begin{exmp} The following maps satisfy hypotheses $H(a)^\prime$:
	\begin{align*}
	a(y)&=\widehat{c}y \quad \mbox{for all $y \in \mathbb{R}^N$, with $\widehat{c}>0$},\\	a(y)&=\begin{cases}|y|^{q-2}y & \mbox{if $|y| \leq 1$,}\\y & \mbox{if $1 <|y|, $} \end{cases} \mbox{ with } 1<q<2, \\
	a(y)&=|y|^{p-2}y +y\quad \mbox{for all $y \in \mathbb{R}^N$, with  $1 < p < +\infty$},\\
	a(y)&=(1+|y|^{2})^\frac{p-2}{2}+y \quad \mbox{for all $y \in \mathbb{R}^N$, with $1 < p < +\infty$},\\
	a(y)&=2y e^{|y|^2}+y\quad \mbox{for all $y \in \mathbb{R}^N$}.
	\end{align*}
\end{exmp}

The new hypotheses on the multivalued perturbation $F(t,x)$ are the following:

\begin{itemize}
	\item[$H(F)_3$:] $F:T  \times \mathbb{R}^N \to P_{f}(\mathbb{R}^N)$ is a multifunction such that
	\begin{enumerate}
		\item[(i)] for all $x \in \mathbb{R}^N$, $t \to F(t,x)$ is graph measurable;
		\item[(ii)] for every $\eta>0$, there exists $k_\eta \in L^\infty(T)_+$ such that 
$$\widehat{\xi}_\eta =\widehat{c}_\eta -\dfrac{\|r\|_\infty}{\widehat{\lambda}_1^{1/2}}-\|k_\eta\|_\infty b^2>0,$$ 
		and
		$$h(F(t,x),F(t,v)) \leq k_\eta(t)|x-v| \quad \mbox{for a.a. $t \in T$, all $|x|,|v| \leq \eta$};$$
		\item[(iii)] the same as hypothesis $H(F)_1$(iii) with $p=2$.\end{enumerate}\end{itemize}
	
\begin{rem}
	Under the above hypotheses $\emptyset \neq \widehat{S} \subseteq \widehat{S}_c \in P_k(C_0^1)$.
\end{rem}	
	
\begin{theorem}
	\label{T8} If hypotheses $H(a)^\prime$, $H(r)$, $H(F)_3$ hold, then $\widehat{S}_c = \overline{\widehat{S}}^{C_0^1}.$
\end{theorem}

\begin{proof}
Let $u \in \widehat{S}_c$. Then we have 
$$-a(u^\prime(t))^\prime -r(t)u^\prime(t)=f(t) \quad \mbox{for a.a. $t \in T$, $u(0)=u(b)=0$,}$$	
with $f \in S^1_{ {\rm \overline{conv} } F(\cdot,u(\cdot))}
$.

Proposition 3.30, p. 185, of Hu-Papageorgiou \cite{Ref5} says that we can find $\{f_n\}_{n \geq 1}\subseteq S^1_{ F(\cdot,u(\cdot))}$ such that $$f_n \xrightarrow{w} f \mbox{ in } L^1_N.$$

Let $v \in W^{1,p}_0$ and $\varepsilon_n \to 0^+$. Consider the multifunction $L^v_n : T \to 2^{\mathbb{R}^N}\setminus \{\emptyset\}$ defined by 
$$L^v_n(t) =\{h \in \mathbb{R}^N : |f_n(t)-h| < \varepsilon_n +d(f_n(t),F(t,v(t))), \, h \in F(t,v(t))\}.$$

Clearly ${\rm Gr}L^v_n  \in \mathcal{L}_T \otimes B(\mathbb{R}^N)$  (recall $\mathcal{L}_T$ is the Lebesgue $\sigma$-field of $T$ and $ B(\mathbb{R}^N)$ the Borel $\sigma$-field of  $\mathbb{R}^N$). By the Yankov-von Neumann-Aumann selection theorem (see Hu-Papageorgiou \cite{Ref5}, Theorem 2.14, p. 158), we can find $h_n :T \to \mathbb{R}^N$, $n \in \mathbb{N}$, a measurable map such that 
$$h_n(t) \in L^v_n(t) \quad \mbox{for a.a. $t \in T$, all $n \in \mathbb{N}$.}$$

Evidently $h_n \in L^1_N$ (that is, $h_n \in S^1_{L^v_n}$).	

We consider the multifunction $G_n:W_0^{1,p} \to 2^{L^1_N}$ defined by
$$G_n(v)=S^1_{L^v_n}.$$

We have just seen that for all $v \in W_0^{1,p}$ and all $n \in \mathbb{N}$, $G_n(v) \neq \emptyset$. Moreover, using Lemma 8.3, p. 239, of Hu-Papageorgiou \cite{Ref5}, we have that 
\begin{align*}
& v \to G_n(v) \mbox{ is lsc},\\
\Rightarrow \quad & v \to \overline{G_n(v)}  \mbox{ is lsc (see \cite{Ref5}, Proposition 2.38, p. 50).}
\end{align*}

Of course this multifunction has decomposable values. So, we can find a continuous map $g_n : W^{1,p}_0 \to L_N^1$, $n \in \mathbb{N}$, such that 
$$g_n(v) \in \overline{G_n(v)}  \quad \mbox{for all $v \in W^{1,p}_0$, all $n \in \mathbb{N}$.}$$

We consider the following nonlinear Duffing system

\begin{equation}
\label{eq22} \begin{cases} - a(v^\prime(t))^\prime -r(t)v^\prime(t)=g_n(v)(t) & \mbox{for a.a. } t \in T,\\ v(0)=v(b)=0, \, n \in \mathbb{N}. &\end{cases}
\end{equation}

This problem has a solution $v_n \in C_0^1$ (see Theorem \ref{T6}).

From \eqref{eq22}, reasoning as in the proof of Proposition \ref{prop5}, we have 
\begin{align}\nonumber
&\left[c_0 - \frac{\|r\|_\infty}{\widehat{\lambda}_1^{1/2}}\right]\|v_n^\prime\|_2^2 \leq \int_0^b (g_n(v_n),v_n)_{\mathbb{R}^N}dt \\ \nonumber & \hskip 3.5cm \leq \int_0^b \left[ a_\varepsilon(t) + ( \theta(t)+\varepsilon)|v_n|^2\right] dt \quad \mbox{(see \eqref{eq17})},\\ \nonumber \Rightarrow \quad & \xi \|v_n^\prime\|_2^2 - \int_0^b \theta(t)|v_n|^2 dt-\frac{\varepsilon}{\widehat{\lambda}_1}\|v_n^\prime\|_2^2 \leq \|a_\varepsilon\|_1,\\ \nonumber \Rightarrow \quad & \left[c_1 - \frac{\varepsilon}{\widehat{\lambda}_1} \right]\|v_n^\prime\|^2_2 \leq \|a_\varepsilon\|_1 \quad \mbox{for all $n \in \mathbb{N}$ (see Lemma \ref{lem2}),}\\ \label{eq23} \Rightarrow \quad & \{v_n\}_{n \geq 1} \subseteq W^{1,2}_0 \mbox{ is bounded (choose $\varepsilon \in (0, \widehat{\lambda}_1c_1)$).}
\end{align}

From \eqref{eq23}, as in the proof of Proposition \ref{prop3} (see the part of the proof from \eqref{eq8} and after), we establish that $\{v_n\}_{n \geq 1} \subseteq C^{1}_0$ is relatively compact. So, we may assume that 
\begin{equation}
\label{eq24} v_n \to v \mbox{ in } C_0^1.
\end{equation}

We have 
$$- a(v_n^\prime)^\prime +a(u^\prime)^\prime -r(t)[v_n^\prime -u^\prime]=g_n(v_n)-f, \quad n \in \mathbb{N}.$$

We act with $v_n-u$ and after integration by parts, we obtain 
\begin{align}\nonumber
& \int_0^b  (a(v_n^\prime)-a(u^\prime),v_n^\prime-u^\prime)_{\mathbb{R}^N}dt - \int_0^b  r(t)(v_n^\prime-u^\prime, v_n -u )_{\mathbb{R}^N}dt \\   = \, &\int_0^b  (g_n(v_n)-f, v_n -u )_{\mathbb{R}^N}dt. \label{eq25}
\end{align}

Let $\eta=\max \{\sup_{n \geq 1}\|v_n\|_\infty,\|u\|_\infty\}>0$. Using hypothesis $H(a)^\prime$ for this $\eta>0$ we have
\begin{equation}
\label{eq26}\widehat{c}_\eta \|v_n^\prime -u^\prime\|^2_2 \leq \int_0^b  (a(v_n^\prime)-a(u^\prime),v_n^\prime-u^\prime)_{\mathbb{R}^N}dt.
\end{equation}

Also, we have 
\begin{align}
\nonumber \left|\int_0^b  r(t)(v_n^\prime-u^\prime, v_n -u )_{\mathbb{R}^N}dt \right| & \leq \|r\|_\infty \| v_n^\prime-u^\prime\|_2 \|v_n-u\|_2\\ \label{eq27}& \leq  \frac{\|r\|_\infty}{\widehat{\lambda}_1^{1/2}}\|v_n^\prime-u^\prime\|^2_2 \quad \mbox{(see \eqref{eq3})}.
\end{align}

Moreover,
\begin{align}
&\nonumber \left|\int_0^b  (g_n(v_n)-f, v_n -u )_{\mathbb{R}^N}dt \right| \\ \label{eq28} \leq & \left|\int_0^b  (f-f_n, v_n -u )_{\mathbb{R}^N}dt 
\right|+  \int_0^b  |(g_n(v_n)-f_n| | v_n -u |dt.\end{align}

Note that
\begin{equation}
\label{eq29}\int_0^b  (f-f_n, v_n -u )_{\mathbb{R}^N}dt \to 0 \mbox{ as } n\to +\infty.
\end{equation}

Also with $\eta>0$ as above, we have
\begin{align}
\nonumber \int_0^b | (g_n(v_n)-f| | v_n -u |dt  & \leq \int_0^b  [\varepsilon_n+h(F(s,v_n),F(s,u))] | v_n -u |ds \\ \label{eq30} & \leq  2 \eta b \varepsilon_n +\int_0^b    k_\eta(t)| v_n -u |^2dt \quad \mbox{for all } n \in \mathbb{N}.\end{align}

Using \eqref{eq29}, \eqref{eq30} in \eqref{eq28}, we obtain
\begin{align}
\nonumber \left|\int_0^b  (g_n(v_n)-f, v_n -u )_{\mathbb{R}^N}dt \right|   &  \leq \varepsilon^\prime_n +\int_0^b    k_\eta(t)| v_n -u |^2dt\\ &  \leq \varepsilon^\prime_n +\|k_\eta\|_\infty b^2\| v^\prime_n -u^\prime \|_2^2 \quad \mbox{with } \varepsilon^\prime_n \to 0^+ \label{eq31}\\ \nonumber & \mbox{(here we have used Jensen's inequality).}\end{align}

Returning to \eqref{eq25} and using \eqref{eq26},  \eqref{eq27} and  \eqref{eq31}, we have 
\begin{align*}
&\left[\widehat{c}_\eta - \frac{\|r\|_\infty}{\widehat{\lambda}_1^{1/2}}-\|k_\eta\|_\infty b^2\right]\|v_n^\prime-u^\prime\|_2^2 \leq \varepsilon_n,\\ \Rightarrow \quad & \widehat{\xi}_\eta \|v^\prime-u^\prime\|_2^2 =0 \quad \mbox{(see \eqref{eq24})},\\ \Rightarrow \quad & v=u.
\end{align*}

So, $v_n \to u$ in $C_0^1$ (see \eqref{eq24}) and $v_n \in \widehat{S}$ for all $n \in \mathbb{N}$. Therefore $\widehat{S}_c = \overline{\widehat{S}}^{C_0^1}.$
\end{proof}

\begin{rem}
Continuing this line of work, it is interesting to know if we can have extremal solutions for the multivalued Duffing system (that is, solutions of \eqref{eq1} when $F(t,x)$ is replaced by ${\rm ext} F(t,x)=$ the extreme points of $F(t,x)$). If such trajectories exist, then we would like to know if they are $C_0^1$-dense in those of the convex problem (strong relaxation). Such a result is of interest in control theory in connection with the ``bang-bang principle''. Results of this kind were proved for a different class of multivalued nonlinear second order systems, by Papageorgiou-Vetro-Vetro \cite{Ref13}.
\end{rem}

\end{document}